\documentclass[onefignum,onetabnum]{siamart251216}

\usepackage{lipsum}
\usepackage{amsfonts}
\usepackage{graphicx}
\usepackage{epstopdf}
\usepackage{algorithmic}
\ifpdf
  \DeclareGraphicsExtensions{.eps,.pdf,.png,.jpg}
\else
  \DeclareGraphicsExtensions{.eps}
\fi

\newsiamremark{remark}{Remark}
\newsiamremark{hypothesis}{Hypothesis}
\crefname{hypothesis}{Hypothesis}{Hypotheses}
\newsiamthm{claim}{Claim}

\headers{Invariance Entropy in the Dust}{Senhan Yao}

\title{Invariance Entropy in the Dust\thanks{Submitted to arXiv as preprint. To be submitted to \emph{SIAM J. Control Optim.}
\funding{This work was supported by the National Natural Science Foundation of China under Grant No. 12288201.}}}

\author{Senhan Yao\thanks{State Key Laboratory of Mathematical Sciences, Academy of Mathematics and Systems Science, Chinese Academy of Sciences, Beijing 100190, People's Republic of China; School of Mathematical Sciences, University of Chinese Academy of Sciences, Beijing 100049, People's Republic of China (\email{yaosenhan@amss.ac.cn}).}}

\usepackage{amsopn}

\makeatletter
\newcommand*{\addFileDependency}[1]{% argument=file name and extension
  \typeout{(#1)}% latexmk will find this if $recorder=0 (however, in that case, it will ignore #1 if it is a .aux or .pdf file etc and it exists! if it doesn't exist, it will appear in the list of dependents regardless)
  \@addtofilelist{#1}% if you want it to appear in \listfiles, not really necessary and latexmk doesn't use this
  \IfFileExists{#1}{}{\typeout{No file #1.}}% latexmk will find this message if #1 doesn't exist (yet)
}
\makeatother

\ifpdf
\hypersetup{
  pdftitle={Invariance Entropy in the Dust},
  pdfauthor={Senhan Yao}
}
\fi

\begin{document}

\maketitle

% REQUIRED
\begin{abstract}
We answer negatively two natural general forms of Kawan’s questions on invariance entropy for control systems, open for more than fifteen years, by a single construction. We show that finite strict invariance entropy need not coincide with ordinary invariance entropy, and that strict invariance entropy need not be lower semicontinuous under Hausdorff perturbations of the initial set. The construction is a continuous-time control system in which a Cantor coordinate stores an infinite symbolic instruction, an exponentially contracting coordinate makes late mismatches geometrically invisible, and a compact matching graph forces exact symbolic agreement. It identifies a source of information complexity not generated by dynamical expansion, but by the persistence of exact viability constraints under thin invariant geometry and by the order of limits in invariance entropy.
\end{abstract}

% REQUIRED

% REQUIRED

\section{Introduction}

\emph{How much information is needed to keep a system inside a prescribed set?} In classical control theory, one usually assumes that the controller has access to the state with unlimited precision and can transmit control commands without restriction \cite{Kawan2013,NairEvansMareelsMoran2004}. In a digitally networked system, this idealization is no longer available \cite{Kawan2013,GarciaKawanYuksel2021}. A controller receiving only finitely many bits over a finite time interval can distinguish only finitely many possible situations and can therefore select only finitely many open-loop control functions \cite{NairEvansMareelsMoran2004,ColoniusKawan2009}. The minimal rate at which the number of necessary controls grows is thus an intrinsic measure of the information required by the control task \cite{ColoniusKawan2009,Kawan2013}.

Motivated by communication-limited control and the search for an intrinsic information rate independent of a particular coder-controller architecture, Nair, Evans, Mareels, and Moran introduced topological feedback entropy for discrete-time systems, following the Adler--Konheim--McAndrew open-cover paradigm \cite{NairEvansMareelsMoran2004,Kawan2013}. Colonius and Kawan subsequently introduced \emph{invariance entropy} for continuous-time systems, in a form closer to the Bowen--Dinaburg spanning-set formulation of topological entropy \cite{ColoniusKawan2009,Kawan2013}. Rather than fixing a communication protocol or feedback law, invariance entropy counts the smallest number of open-loop controls needed to keep every initial state in \(K\) inside, or arbitrarily close to, a controlled invariant set \(Q\) up to time \(T\), and measures their exponential growth rate as \(T\to\infty\) \cite[Sec.~1 and Sec.~3]{ColoniusKawan2009}. The analogy with topological entropy is therefore structural: the spanning objects are controls rather than orbit segments, and the controlled dynamics is not an autonomous system on the state space \cite{Kawan2013,ColoniusKawanNair2013}. This structural analogy was made precise by Huang and Zhong, who introduced Carath\'eodory--Pesin structures associated with control systems and showed that both invariance entropy and topological feedback entropy are quantities generated by special such structures, thereby revealing their dimension-type character \cite[Abstract and Sec.~1]{HuangZhong2018}.

Since its introduction by Colonius and Kawan \cite[Sec.~1]{ColoniusKawan2009}, continuous-time invariance entropy has been related to topological feedback entropy and minimal data rates \cite{ColoniusKawanNair2013,Kawan2013}, bounded in terms of Lipschitz constants, dimension, divergence, escape rates, and volume growth \cite{Kawan2011UpperLower,Kawan2011LowerBounds,Kawan2013}, and characterized by exact or variational formulas for linear systems, control sets, and uniformly hyperbolic chain control sets \cite{ColoniusKawan2009,Kawan2011ControlSets,DaSilvaKawan2016Hyperbolic}. Extensions include invariance pressure and measure-theoretic variants \cite[Abstract and Sec.~1]{ColoniusCossichSantana2018,WangHuangSun2019}, while related entropy notions have been developed for state estimation, exponential and practical stabilization, and the robustness of critical bit rates under perturbations \cite{Kawan2017StateEstimation,Colonius2012MinimalBitRates,ColoniusHamzi2021,DaSilvaKawan2016Robustness}.

These developments show that information complexity in control systems is not exhausted by a single exponential growth rate \cite[Abstract and Sec.~1]{HuangZhong2018,WangHuangSun2019}. In the positive-rate theory, it is often traced to dynamical instability: unstable eigenvalues, positive Lyapunov exponents, and volume expansion force an increasing number of trajectories, controls, or messages to be distinguished \cite{ColoniusKawan2009,DaSilvaKawan2016Hyperbolic,Kawan2011LowerBounds}. The dimension-type viewpoint has recently been pushed further to control systems with zero invariance entropy, through Bowen and packing \(\alpha\)-invariance entropies, invariance entropy dimensions, and corresponding measure-theoretic variational principles \cite[Abstract and Sec.~1]{ChenHuangZhong2026}.

The phenomenon studied here is different again. It comes from the distinction between keeping trajectories exactly inside a controlled invariant set and keeping them arbitrarily close to that set. Although exact invariance and approximate invariance are clearly different over a fixed finite time interval, it is not evident whether this difference can survive after one passes to the long-time information rate. This leads to two natural questions:
\begin{itemize}
\item If the exact invariance task has a finite asymptotic information rate, must it have the same rate as the corresponding approximate invariance task?
\item Do these asymptotic information rates behave continuously, or at least semicontinuously, when the set of initial states and the target invariant set are perturbed in the Hausdorff metric?
\end{itemize}
Both questions were explicitly raised in Kawan's early development of invariance entropy. Open Question~2.1.15 asks whether finite strict invariance entropy forces equality between strict and ordinary invariance entropy, while Open Question~2.2.14 asks whether ordinary or strict invariance entropy depends continuously on the initial set and/or the invariant target set in some sense and under some condition \cite[p.~49 and Open Question~2.1.15 on p.~58; p.~66 and Open Question~2.2.14]{KawanThesis2009}. Both questions are motivated by a common stability intuition: geometric closeness at every fixed finite horizon should be reflected in the corresponding infinite-time information rates.

\emph{A single construction, based on a Cantor dust, gives definitive negative answers to two natural general forms of Kawan's longstanding questions: finite strict invariance entropy need not coincide with ordinary invariance entropy, and strict invariance entropy need not be lower semicontinuous under Hausdorff perturbations of the initial set.} More precisely, Theorem~\ref{thm:main-separation} answers Kawan's first open question \cite[p.~49 and Open Question~2.1.15 on p.~58]{KawanThesis2009} negatively, while Theorem~\ref{thm:hausdorff-failure} gives a negative answer to the unconditional Hausdorff lower semicontinuity version of Kawan's second open question \cite[p.~66 and Open Question~2.2.14]{KawanThesis2009}.

\emph{More fundamentally, the example exhibits a source of information complexity that is not generated by dynamical expansion.} In our system, the positive information rate is not generated by local exponential separation or volume expansion; instead, the information lies in determining which open-loop control must be followed in order to remain exactly on a geometrically thin invariant set. The geometry is arranged so that one state variable contracts, and hence the visible effect of later control choices becomes exponentially small.

Distinct control requirements therefore become arbitrarily close in the state-space geometry without becoming interchangeable for exact invariance. Every positive tolerance eventually conceals these distinctions, and the finite Cantor skeletons used in the approximation retain only finitely many independent symbolic choices, whereas the limiting Cantor dust preserves an inexhaustible sequence of exact control requirements. Thus the information required for exact invariance does not arise from instability: it survives in the dust.

\section{Kawan's questions}

To formulate Kawan's questions in the terminology used throughout the paper, we recall the standard spanning-set definitions of Colonius and Kawan \cite{ColoniusKawan2009,Kawan2013}. The exact-versus-ordinary equality problem is Kawan's Open Question~2.1.15, and the continuity problem with respect to \(K\) and/or \(Q\) is Kawan's Open Question~2.2.14 \cite[p.~58 and p.~66]{KawanThesis2009}.

The relevant distinction is between exact containment and containment in an outer neighborhood. \emph{Strict invariance entropy} requires the controlled trajectory to remain exactly in \(Q\). The \emph{standard invariance entropy}, which we also call the \emph{ordinary} or \emph{non-strict invariance entropy}, requires only that the trajectory remain in an arbitrarily small outer neighborhood of \(Q\), with the tolerance sent to zero after the long-time exponential growth rate has been taken. Throughout, \(\log\) denotes the natural logarithm.

Throughout this section, let \(M\) be a metric state space with metric \(d\), and let \(U\subseteq \mathbb R^m\) be a compact control range. An \emph{admissible control function} is a Lebesgue measurable map with values in \(U\), and the set of all admissible controls is denoted by
\[
\mathcal U
:=
\bigl\{
u:[0,\infty)\to U:
u \text{ is measurable and } u(t)\in U
\text{ for a.e. } t\ge 0
\bigr\}.
\]
We reserve \(U\) for the compact control range and \(\mathcal U\) for admissible open-loop controls. Controls which agree almost everywhere are identified; in particular, all equalities of controls are understood a.e. For \(x\in M\) and \(u\in\mathcal U\), the corresponding trajectory is denoted by \(\varphi(t,x,u)\) and is understood in the Carath\'{e}odory sense on the time intervals under consideration. Thus, when a finite family of admissible controls is used below, one control from that family may be selected according to the initial state, but after this selection the control is open-loop on the time interval under consideration. A compact set \(Q\subseteq M\) is called \emph{weakly invariant} if, for every \(x\in Q\), there exists \(u\in\mathcal U\) such that
\[
    \varphi(t,x,u)\in Q
    \qquad\text{for all }t\ge0 .
\]
Here \emph{weakly invariant} is synonymous with the standard term \emph{controlled invariant} used in much of the invariance-entropy literature; the adjective \emph{weakly} only stresses that the condition is existential in the control. For \(\varepsilon>0\), we write
\[
    N_\varepsilon(Q)
    :=
    \{x\in M:\operatorname{dist}(x,Q)<\varepsilon\},
    \qquad
    \operatorname{dist}(x,Q):=\inf_{q\in Q}d(x,q).
\]

More precisely, let \(K\subseteq Q\) be compact and \(T>0\). A finite set \(S\subseteq\mathcal U\) is called \emph{\((T,K,Q)\)-spanning} if for every \(x\in K\) there exists \(u\in S\) such that
\[
    \varphi(t,x,u)\in Q
    \qquad\text{for all }t\in[0,T].
\]
The smallest cardinality of such a family is denoted by \(r_{\mathrm{inv}}^*(T,K,Q)\), with the convention that this number is \(\infty\) if no finite spanning family exists. We use the convention \(\log\infty=\infty\). The strict invariance entropy is
\begin{equation}\label{eq:02.01}
h_{\mathrm{inv}}^*(K,Q) := \limsup_{T\to\infty}
\frac{1}{T}\log r_{\mathrm{inv}}^*(T,K,Q).
\end{equation}
For \(\varepsilon>0\), a finite set \(S\subseteq\mathcal U\) is called \emph{\((T,\varepsilon,K,Q)\)-spanning} if for every \(x\in K\) there exists \(u\in S\) such that
\[
    \varphi(t,x,u)\in N_\varepsilon(Q)
    \qquad\text{for all }t\in[0,T].
\]
The smallest cardinality of such a family is denoted by \(r_{\mathrm{inv}}(T,\varepsilon,K,Q)\), again with the convention that this number is \(\infty\) if no finite spanning family exists. Then
\[
    h_{\mathrm{inv}}(\varepsilon,K,Q)
    :=
    \limsup_{T\to\infty}
    \frac{1}{T}\log r_{\mathrm{inv}}(T,\varepsilon,K,Q),
\]
and the invariance entropy is
\begin{equation}\label{eq:02.02}
h_{\mathrm{inv}}(K,Q):= \lim_{\varepsilon\downarrow0}
h_{\mathrm{inv}}(\varepsilon,K,Q).
\end{equation}
The limit exists in \([0,\infty]\), since \(h_{\mathrm{inv}}(\varepsilon,K,Q)\) is monotone nondecreasing as \(\varepsilon\downarrow0\). Crucially, the infinite-time growth rate is taken before the tolerance is sent to zero.

Exact invariance is the stronger requirement, and therefore 
\begin{equation}\label{eq:02.03}
h_{\mathrm{inv}}(K,Q)
\leq
h_{\mathrm{inv}}^*(K,Q).
\end{equation}
It is much less clear, however, whether the two quantities can differ when the strict entropy is finite. For every fixed finite horizon, continuous dependence on initial conditions under a fixed control seems to bring exact and approximate invariance close together. Moreover, the tolerance in the definition of \(h_{\mathrm{inv}}(K,Q)\) is eventually sent to zero. These observations motivate the following natural possibility, explicitly left open in the early theory \cite[p.~49 and Open Question~2.1.15 on p.~58]{KawanThesis2009}:
\begin{equation}\label{eq:02.04}
h_{\mathrm{inv}}^*(K,Q)<\infty
\quad\Longrightarrow\quad
h_{\mathrm{inv}}(K,Q)=h_{\mathrm{inv}}^*(K,Q).
\end{equation}

A second question concerns the stability of invariance entropy under perturbations of the sets that define the control task. This is also an explicit issue in Kawan's thesis: Open Question~2.2.14 asks whether \(h_{\mathrm{inv}}^{*}(K,Q)\) depends continuously on \(K\) and/or \(Q\) ``in some sense and under some condition'' \cite[p.~66 and Open Question~2.2.14]{KawanThesis2009}. The present paper tests this open continuity problem in the compact-pair setting. For compact sets, the Hausdorff metric provides the natural topology in which to formulate this question. Hausdorff convergence, however, controls only the static geometry of the sets; it does not directly compare the controls required to keep trajectories viable over increasingly long time intervals. It is therefore natural to ask whether geometric convergence nevertheless implies at least upper or lower semicontinuity of the associated asymptotic information rates.

\section{Main Construction and Mechanism} 

We now describe the construction and the mechanism behind both negative answers. The purpose of this section is to make the geometry and the information structure transparent. Precise definitions and proofs of the properties stated below are given in the subsequent section.

The system has three state variables, ordered as \((z,c,y)\). Variable \(c\) stores an infinite symbolic instruction, \(z\) progressively suppresses the geometric effect of that instruction, and \(y\) records whether the applied control follows it. The target set is designed so that exact invariance forces the control to follow the prescribed symbolic sequence at every stage, even when a wrong control choice would produce only an arbitrarily small displacement from the set.

\subsection{Frozen Code}

We begin by building an infinite symbolic instruction into the initial
state. Let \(C\) be the standard middle-third Cantor set. For the purposes of the construction, one may think of \(C\) as a geometric realization of the space of all infinite binary sequences. Thus each point \(c\in C\) carries a sequence
\begin{equation}\label{eq:03.01}
    \sigma(c)
=
\bigl(\sigma_1(c),\sigma_2(c),\ldots\bigr),
\qquad
\sigma_j(c)\in\{-1,1\}.
\end{equation}
For convenience, and especially for readers less familiar with the Cantor set, the precise symbolic coding of \(C\) and its compactness are recalled in Subsection~\ref{subsection:cantor_set}.

To store the symbolic code permanently in the state, we introduce a state coordinate
\(c\) and prescribe its dynamics by \(\dot{c} = 0\). We call \(c\) the \emph{Cantor coordinate}, although as a state variable it ranges over an open interval containing the Cantor set \(C\); only the initial set and the target set constructed below restrict this coordinate to \(C\). Once the initial condition selects \(c\), the associated sequence
\(\sigma(c)\) remains frozen throughout the evolution. 

At this
stage the sequence imposes no restriction on the control; the target set
introduced below will convert this frozen code into a time-dependent matching
requirement.

\subsection{Exponential Fading}

The coordinate \(z\) satisfies \(\dot z=-z, z(0)=1\), and hence \(z(t)=e^{-t}\). This variable has two roles. First, it records the passage of time, since \(t=-\log z(t)\). Second, it provides a continuously shrinking geometric scale. A discrepancy occurring at a late time is multiplied by a much smaller value of \(z\) than the same discrepancy occurring near the initial time.

The third coordinate \(y\) is governed by \(\dot y=-y+zu, u\in[-1,1]\). Thus the control enters the \(y\)-dynamics through the decaying factor \(z(t)=e^{-t}\). Different controls may remain completely distinct as functions of time, while the differences they produce in the state become progressively smaller.

At this point the full controlled system has been specified on the Euclidean state space \(M=\mathbb R^3\), with coordinates \((z,c,y)\) and the Euclidean metric. The control range is \(U=[-1,1]\), and \(\varphi(t,x,u)\) denotes the Carath\'{e}odory solution corresponding to the admissible control \(u\in\mathcal U\):
\begin{equation}\label{eq:03.02}
    \dot z=-z,\qquad
\dot c=0,\qquad
\dot y=-y+zu,
\qquad u\in[-1,1].
\end{equation}
For this system, admissible controls are precisely the elements of
\(L^\infty_{\mathrm{loc}}([0,\infty),[-1,1])\), modulo equality a.e.

We now choose the set of initial states
\begin{equation}\label{eq:03.03}
K:=\{(1,c,0):c\in C\}.
\end{equation}
Since the set \(K\) is a continuous image of the compact set \(C\), \(K\) is compact. Thus every trajectory starts with the same visible coordinates \(z=1\) and
\(y=0\); the initial states differ only in the frozen symbolic code carried by
\(c\). This isolates the source of information complexity: no distinction is
encoded in the initial values of the contracting and controlled coordinates,
and all future control requirements originate from the Cantor coordinate.

\medskip

\subsection{Matching Graph}

The Cantor coordinate stores an infinite symbolic instruction, but so far
this instruction has no effect on the motion. We now design the target set
so that remaining in it forces the control to ``follow the stored symbols''. 

For each \(c\in C\), the sequence \(\sigma(c)\) is interpreted as an infinite control schedule: The first symbol prescribes
the control during the first unit time interval, the second symbol prescribes
the control during the second interval, and so on. Accordingly, we define
the piecewise constant reference signal
\begin{equation}\label{eq:03.04}
a(t,c)=\sigma_j(c)
\qquad
\text{for }t\in[j-1,j),\quad j\geq1.
\end{equation}
Thus, following the code stored in \(c\) means choosing
\begin{equation}\label{eq:03.05}
u(t)=a(t,c)
\end{equation}
for almost every \(t\). The phrase ``for almost every \(t\)'' is essential:
changing a measurable control at isolated times does not change the
corresponding trajectory.

The role of the target set \(Q\) is to test this matching condition
geometrically. For each pair \((z,c) \in [0,1] \times C\), the graph admits exactly one value of the
third coordinate \(y\). We define \(Q\) as a graph over the \((z,c)\)-coordinates, so that once \((z,c)\) is fixed there is only one admissible value of \(y\). This single allowed value is chosen to be exactly the \(y\)-coordinate produced by the perfectly matching control. Thus a trajectory following the prescribed signal remains on the graph, while any mismatch on a set of positive measure prevents the trajectory from satisfying the graph relation for all times in the interval. The construction of \(Q\) is given in Subsection~\ref{subsection:matching_properties}. Moreover, \(K \subseteq Q\) and \(Q\) is compact.

\subsection{Strict and Approximate Invariance}

The graph \(Q\) has been designed so that exact invariance is equivalent to exact symbolic matching. If an initial state carries the code \(c\), then keeping its trajectory in \(Q\) forces the chosen open-loop control function \(u(\cdot)\) to reproduce the symbols of \(\sigma(c)\) on the corresponding unit time intervals. Thus, to keep every initial state in \(K\) inside \(Q\) up to time \(T\), a finite set of admissible open-loop control functions must contain enough different functions to reproduce the first \(\lceil T\rceil\) symbols of every code carried by \(K\).

Because \(K\) contains the full Cantor family of codes, every word of length \(\lceil T\rceil\) over the alphabet \(\{-1,1\}\) occurs as the initial block of some code. Two different such words cannot in general be handled by the same open-loop control function. On the first unit interval where they differ, that function would have to take the value \(1\) in order to match one code and the value \(-1\) in order to match the other. This is impossible. Hence at least one control function is needed for each possible word of length \(\lceil T\rceil\). Conversely, this number of control functions is sufficient. For each word of length \(\lceil T\rceil\), choose one open-loop control function that reproduces this word on the corresponding unit intervals, and assign to each initial state the function whose word agrees with the first \(\lceil T\rceil\) symbols of its code. Then the trajectory remains on \(Q\) up to time \(T\). Therefore
\begin{equation}\label{eq:03.06}
r_{\mathrm{inv}}^*(T,K,Q)=2^{\lceil T\rceil}.
\end{equation}
Every additional unit of time introduces one further binary distinction that exact invariance must preserve. Consequently,
\begin{equation}\label{eq:03.07}
h_{\mathrm{inv}}^*(K,Q)
=
\limsup_{T\to\infty}
\frac{1}{T}\log 2^{\lceil T\rceil}
=
\log 2.
\end{equation}

Approximate invariance treats the same symbolic errors differently because it measures their geometric effect rather than their exact occurrence. Suppose that a control reproduces the prescribed symbols during the first \(N\) unit intervals and is arbitrary afterwards. Then every possible mismatch between the chosen control and the prescribed signal occurs only after time \(N\). By that time, the factor \(z(t)=e^{-t}\) has already become uniformly small, and any later discrepancy enters the \(y\)-equation only after being multiplied by this small factor. Consequently, the trajectory may leave the matching graph after time \(N\), but it can do so only by a uniformly small amount. As shown in the detailed proof in Subsection~\ref{subsection:fading_estimate}, the distance from the trajectory to \(Q\) is bounded by a quantity of order \(e^{-N}\), independently of how long the trajectory is followed.

This is where the order of limits in the definition of \(h_{\mathrm{inv}}(K,Q)\) becomes decisive. For a fixed tolerance \(\varepsilon>0\), choose \(N\) so large that all errors occurring after the first \(N\) unit intervals remain within the \(\varepsilon\)-neighborhood of \(Q\). Once this \(N\) is fixed, the time horizon \(T\) is allowed to go to infinity, but no new symbols beyond the first \(N\) have to be distinguished at that resolution.

At this fixed tolerance, one does not need to distinguish the whole infinite code. It is enough to distinguish the finite prefix that remains visible at scale \(\varepsilon\). More concretely, for each word
\[
\omega=(\omega_1,\ldots,\omega_N)\in\{-1,1\}^N,
\]
choose one open-loop control function \(u_\omega\) which takes the value \(\omega_j\) on the \(j\)-th unit time interval, \(j=1,\ldots,N\), and is defined arbitrarily afterwards. This gives a finite catalogue of \(2^N\) admissible control functions.

Now take any initial state in \(K\), carrying some code \(c\). Let \(\omega\) be the first \(N\) symbols of \(\sigma(c)\), and assign to this initial state the control function \(u_\omega\). Up to time \(N\), the assigned control matches the prescribed signal exactly, so the trajectory stays on the graph \(Q\). After time \(N\), the control may no longer match the remaining symbols of \(\sigma(c)\), but all such later mismatches occur at the exponentially small scale \(z(t)=e^{-t}\). By the preceding estimate, the resulting deviation from \(Q\) remains below \(\varepsilon\), independently of the final time horizon \(T\). Therefore the same catalogue of \(2^N\) control functions works for every \(T\).

Equivalently, for this fixed tolerance \(\varepsilon\), the minimal cardinality \(r_{\mathrm{inv}}(T,\varepsilon,K,Q)\) of a set of admissible open-loop control functions that keeps every trajectory within the \(\varepsilon\)-neighborhood of \(Q\) satisfies
\begin{equation}\label{eq:03.09}
    r_{\mathrm{inv}}(T,\varepsilon,K,Q)\le 2^N
\quad\text{for all }T>0.
\end{equation}
Since \(r_{\mathrm{inv}}(T,\varepsilon,K,Q)\ge 1\), it follows that
\begin{equation}\label{eq:03.10}
\limsup_{T\to\infty}
\frac{1}{T}\log r_{\mathrm{inv}}(T,\varepsilon,K,Q)
=
0.
\end{equation}
Only after this infinite-time growth rate has been computed do we let \(\varepsilon\downarrow0\). Decreasing the tolerance may force the prefix length \(N\) to increase, but it does not create exponential growth in \(T\). Therefore
\begin{equation}\label{eq:03.11}
h_{\mathrm{inv}}(K,Q)
=
\lim_{\varepsilon\downarrow0}
\limsup_{T\to\infty}
\frac{1}{T}\log r_{\mathrm{inv}}(T,\varepsilon,K,Q)
=
0.
\end{equation}
We record the conclusion as the first main theorem of the paper.

\begin{theorem}[Separation of exact and approximate invariance entropy]\label{thm:main-separation}
For the control system \eqref{eq:03.02} on \(M=\mathbb R^3\), equipped with the Euclidean metric and control range \(U=[-1,1]\), and for the compact sets \(K\) and \(Q\) defined in \eqref{eq:03.03} and \eqref{eq:04.16}, respectively, one has \(K\subseteq Q\), \(Q\) is weakly invariant, \(h_{\mathrm{inv}}(K,Q)=0\), and \(h_{\mathrm{inv}}^*(K,Q)=\log 2\). In particular, finite strict invariance entropy need not coincide with ordinary invariance entropy.
\end{theorem}

\subsection{Hausdorff Approximation}

The same construction also gives a negative answer to the Hausdorff stability question. The guiding observation is simple: Hausdorff convergence detects geometric closeness, but strict invariance entropy detects how many independent control choices remain present over arbitrarily long time intervals. A finite approximation may be very close to the Cantor dust as a set, while having already lost the infinite symbolic tail responsible for the positive strict entropy.

We construct the approximating sets from the finite symbolic stages of the
Cantor set. For each word
\(\omega=(\omega_1,\ldots,\omega_n)\in\{-1,1\}^n\), let
\[
    \omega^+
    :=
    (\omega_1,\ldots,\omega_n,1,1,1,\ldots),
    \qquad
    c_\omega^+:=\pi(\omega^+).
\]
Here \(\pi\) denotes the standard symbolic coding map of the middle-third Cantor set, defined in \eqref{eq:04.02}. Thus \(c_\omega^+\in C\) is the Cantor point whose code begins with \(\omega\)
and then has the constant tail \(1,1,1,\ldots\). Define
\begin{equation}\label{eq:03.12}
    C_n:=\{c_\omega^+:\omega\in\{-1,1\}^n\},
    \qquad
    K_n:=\{(1,c,0):c\in C_n\}.
\end{equation}
Thus \(C_n\) contains exactly one representative from each symbolic cylinder
of length \(n\). Equivalently, \(C_n\) is the finite right-endpoint skeleton of
the \(n\)-th stage of the Cantor construction. In particular, \(C_n\subseteq C\)
is finite, and hence \(K_n\) is compact. The fact that \(C_n\to C\) in the
Hausdorff metric, and consequently \(K_n\to K\), is verified in
Subsection~\ref{subsection:hausdorff_convergence}.

The strict entropy, however, behaves very differently. For a fixed \(n\), the set \(K_n\) contains only \(n\) independent binary choices; all symbols after time \(n\) are fixed to \(1\). During the first \(n\) unit time intervals, different initial states may require different open-loop control functions, according to their different symbolic prefixes. After time \(n\), all codes represented in \(C_n\) have the same tail. Hence no new independent control choice appears after the first \(n\) intervals. 

We now count the control functions needed for exact invariance of \(K_n\). The set \(K_n\) has only \(2^n\) initial states, one for each word \(\omega\in\{-1,1\}^n\). Moreover, once \(\omega\) is fixed, the entire symbolic code of the corresponding point \(c_\omega^+\) is fixed as well: it begins with \(\omega\) and then has the constant tail \(1,1,1,\ldots\). For each word \(\omega\), choose one open-loop control function \(u_\omega\) that follows this entire code: it reproduces \(\omega\) during the first \(n\) unit intervals and then follows this constant tail afterwards. Assign \(u_\omega\) to the initial state \((1,c_\omega^+,0)\). By the defining property of the matching graph \(Q\), a trajectory whose control follows the code stored in its Cantor coordinate remains exactly on \(Q\). Hence each initial state in \(K_n\) can be kept in \(Q\) for all time by one of the \(2^n\) functions in this catalogue. Therefore the same catalogue works for every time horizon \(T>0\). In other words,
\begin{equation}\label{eq:03.13}
r_{\mathrm{inv}}^*(T,K_n,Q)\leq 2^n \qquad \text{for all }T>0.
\end{equation}
This bound is uniform in \(T\). Since also \(r_{\mathrm{inv}}^*(T,K_n,Q)\geq1\), its exponential growth rate in time is zero: 
\begin{equation}\label{eq:03.14}
h_{\mathrm{inv}}^*(K_n,Q) = \limsup_{T\to\infty} \frac{1}{T} \log r_{\mathrm{inv}}^*(T,K_n,Q) = 0 \qquad \text{for every }n.
\end{equation}

The Hausdorff limit restores what every finite approximation has lost. The full Cantor set \(C\) has no last free symbol. On every successive unit time interval, the family of initial states in \(K\) presents a new independent binary choice. From the preceding subsection, \( h_{\mathrm{inv}}^*(K,Q)=\log 2\). Thus the strict invariance entropy jumps upward at the Hausdorff limit: 
\begin{equation}\label{eq:03.15}
    K_n\longrightarrow K, \qquad h_{\mathrm{inv}}^*(K_n,Q)=0 \quad\text{for all }n, \qquad h_{\mathrm{inv}}^*(K,Q)=\log 2.
\end{equation}
We record this conclusion as the second main theorem.

\begin{theorem}[Failure of Hausdorff lower semicontinuity]\label{thm:hausdorff-failure}
For the same control system \eqref{eq:03.02}, the target set \(Q\) in \eqref{eq:04.16}, the limiting initial set \(K\) in \eqref{eq:03.03}, and the approximating initial sets \(K_n\) in \eqref{eq:03.12}, the sets \(K_n\subseteq Q\) converge to \(K\) in the Hausdorff metric, but \(\liminf_{n\to\infty} h_{\mathrm{inv}}^*(K_n,Q) < h_{\mathrm{inv}}^*(K,Q)\). More precisely, \(h_{\mathrm{inv}}^*(K_n,Q)=0\) for every \(n\), while \(h_{\mathrm{inv}}^*(K,Q)=\log 2\). In particular, strict invariance entropy is not lower semicontinuous with respect to unconditional Hausdorff perturbations of the initial set, with the target set \(Q\) fixed.
\end{theorem}

\section{Verification of Facts}

The preceding section described the construction and derived its consequences at the level of the main mechanism. We now verify the auxiliary facts that were used there without proof.

\subsection{Cantor Set}
\label{subsection:cantor_set}

We use a symbolic coding of the standard middle-third Cantor set. Let
\begin{equation}\label{eq:04.01}
    \Sigma:=\{-1,1\}^{\mathbb N}
\end{equation}
be endowed with the product topology. Define
\begin{equation}\label{eq:04.02}
    \pi:\Sigma\longrightarrow [0,1],
    \qquad
    \pi(\omega)
    :=
    \sum_{j=1}^{\infty}\frac{\omega_j+1}{3^j},
    \qquad
    \omega=(\omega_1,\omega_2,\ldots)\in\Sigma .
\end{equation}
Thus the symbol \(-1\) corresponds to the ternary digit \(0\), while the
symbol \(1\) corresponds to the ternary digit \(2\). We define
\begin{equation}\label{eq:04.03}
    C:=\pi(\Sigma).
\end{equation}
Equivalently, \(C\) consists precisely of those points of \([0,1]\) whose
ternary expansion uses only the digits \(0\) and \(2\).

The map \(\pi\) is continuous, since the defining series converges uniformly.
It is also injective. Indeed, if \(\omega,\omega'\in\Sigma\) are distinct, let
\(m\) be the first index such that \(\omega_m\neq \omega'_m\). Then
\[
    |\pi(\omega)-\pi(\omega')|
    \geq
    \frac{2}{3^m}
    -
    \sum_{j=m+1}^{\infty}\frac{2}{3^j}
    =
    \frac{1}{3^m}
    >
    0.
\]
Since \(\Sigma\) is compact and \([0,1]\) is Hausdorff, \(\pi\) is a
homeomorphism from \(\Sigma\) onto \(C\). In particular, \(C\) is compact.

We now define the symbolic code carried by a point of the Cantor set. Since
\(\pi:\Sigma\to C\) is a homeomorphism, for every \(c\in C\) there exists a
unique sequence \(\omega\in\Sigma\) such that \(\pi(\omega)=c\). We set
\begin{equation}\label{eq:04.05}
    \sigma(c)
    :=
    \pi^{-1}(c)
    =
    \bigl(\sigma_1(c),\sigma_2(c),\ldots\bigr)
    \in\{-1,1\}^{\mathbb N}.
\end{equation}
Thus \(\sigma_j(c)\) denotes the \(j\)-th symbol in the unique symbolic code of
\(c\).

We next distinguish between the interval stages of the Cantor construction and
the finite approximating sets used later. Let
\begin{equation}\label{eq:04.06}
    \mathcal C_0=[0,1],
    \qquad
    \mathcal C_{n+1}
    :=
    \frac13\mathcal C_n
    \cup
    \left(\frac23+\frac13\mathcal C_n\right).
\end{equation}
Then
\begin{equation}\label{eq:04.07}
    C=\bigcap_{n=0}^{\infty}\mathcal C_n.
\end{equation}
Thus \(\mathcal C_n\) denotes the usual \(n\)-th closed stage of the
middle-third Cantor construction.

For a word \(\omega=(\omega_1,\ldots,\omega_n)\in\{-1,1\}^n\), define
\begin{equation}\label{eq:04.08}
    I_\omega
    :=
    \left[
        \sum_{j=1}^{n}\frac{\omega_j+1}{3^j},
        \sum_{j=1}^{n}\frac{\omega_j+1}{3^j}+3^{-n}
    \right].
\end{equation}
Then
\begin{equation}\label{eq:04.09}
    \mathcal C_n
    =
    \bigcup_{\omega\in\{-1,1\}^n} I_\omega .
\end{equation}
The interval \(I_\omega\) is the \(n\)-th level Cantor interval corresponding
to the symbolic word \(\omega\).

We now define the finite approximating sets. For each
\(\omega=(\omega_1,\ldots,\omega_n)\in\{-1,1\}^n\), set
\begin{equation}\label{eq:04.11}
    \omega^+
    :=
    (\omega_1,\ldots,\omega_n,1,1,1,\ldots)\in\Sigma,
    \qquad
    c_\omega^+:=\pi(\omega^+).
\end{equation}
Equivalently, \(c_\omega^+\) is the right endpoint of the Cantor interval
\(I_\omega\). We define
\begin{equation}\label{eq:04.12}
    C_n:=\{c_\omega^+:\omega\in\{-1,1\}^n\}.
\end{equation}
Thus \(C_n\) is the finite right-endpoint skeleton of the \(n\)-th Cantor stage
\(\mathcal C_n\). In particular,
\[
    C_n\subseteq C,
    \qquad
    C_n \text{ is finite},
    \qquad
    |C_n|=2^n .
\]

\subsection{Matching Graph}
\label{subsection:matching_properties}

We first define the accumulated reference signal by
\begin{equation}\label{eq:04.14}
q(t,c)
:=
\int_0^t a(s,c)\,ds .
\end{equation}
If the control follows the stored instruction, then the \(y\)-coordinate starting from \(y(0)=0\) is \(e^{-t}q(t,c)\). Since \(z(t)=e^{-t}\), equivalently \(t=-\log z(t)\), this matching value can be written as a function of \((z,c)\):
\begin{equation}\label{eq:04.15}
F(z,c)
:=
zq(-\log z,c),
\qquad z>0,
\qquad
F(0,c)=0.
\end{equation}
This motivates the definition of the target graph
\begin{equation}\label{eq:04.16}
Q
:=
\bigl\{
(z,c,F(z,c)):
z\in[0,1],\ c\in C
\bigr\}.
\end{equation}
We first verify that this definition indeed gives a continuous graph, despite the jumps of the reference signal in time.

\begin{proposition}[Continuity of the matching graph]
\label{prop:F_continuous}
The map \(F:[0,1]\times C\to\mathbb R\) defined by \eqref{eq:04.15} is continuous.
\end{proposition}

\begin{proof}
For every \(j\geq1\), the coordinate map
\[
    \sigma_j:C\longrightarrow \{-1,1\}
\]
is continuous. Indeed, \(\sigma_j\) is the composition of the homeomorphism
\(\pi^{-1}:C\to\Sigma\) with the \(j\)-th coordinate projection on the product
space \(\Sigma\). Equivalently, the sets on which finitely many symbols are
prescribed are Cantor cylinders, and these cylinders are clopen in \(C\).

We next show that \(q\) is continuous on every compact time interval. Fix
\(R>0\). If \(m\in\mathbb N\cup\{0\}\) and \(s\in[m,m+1]\), then, by the
definition of \(a\),
\[
    q(s,c)
    =
    \sum_{j=1}^{m}\sigma_j(c)
    +(s-m)\sigma_{m+1}(c),
\]
with the convention that the empty sum is zero when \(m=0\). The right-hand
side is continuous in \((s,c)\) on \([m,m+1]\times C\), because it involves only
finitely many continuous coordinate maps \(\sigma_j\). On overlaps of adjacent
strips, namely at integer values of \(s\), the two formulas agree. Since only
finitely many strips meet \([0,R]\times C\), the pasting lemma gives continuity
of \(q\) on \([0,R]\times C\).

It follows immediately that \(F\) is continuous at every point with \(z>0\),
since there \(s=-\log z\) ranges over a finite time interval locally and
\[
    F(z,c)=zq(-\log z,c)
\]
is a composition of continuous maps.

It remains to check continuity at \(z=0\). For all \(s\geq0\) and all \(c\in C\),
\[
    |q(s,c)|
    \leq
    \int_0^s |a(r,c)|\,dr
    \leq s,
\]
because \(a(r,c)\in\{-1,1\}\). Therefore, for \(0<z\leq1\),
\[
    |F(z,c)|
    =
    z|q(-\log z,c)|
    \leq
    z|\log z|.
\]
The last quantity tends to \(0\) as \(z\downarrow0\), uniformly in \(c\in C\).
Since \(F(0,c)=0\) for every \(c\in C\), this proves continuity at all points of
\(\{0\}\times C\). Hence \(F\) is continuous on \([0,1]\times C\).
\end{proof}

Since \(F\) is continuous on \([0,1]\times C\) and \(C\) is compact, the graph \(Q\) is compact; moreover \(F(1,c)=q(0,c)=0\), so \(K\subseteq Q\).

We also record explicitly that the target set itself, not only the initial set
\(K\), is weakly invariant.

\begin{proposition}[Weak invariance of \(Q\)]
\label{prop:Q_weakly_invariant}
The set \(Q\) is weakly invariant for the control system \eqref{eq:03.02}. More
precisely, for every initial point \(x_0\in Q\), there exists an admissible
control \(u\in L^\infty_{\mathrm{loc}}([0,\infty),[-1,1])\) such that the
corresponding solution remains in \(Q\) for all \(t\geq0\).
\end{proposition}

\begin{proof}
Let
\[
    x_0=(z_0,c,F(z_0,c))\in Q,
    \qquad z_0\in[0,1],\quad c\in C.
\]
First assume that \(z_0>0\), and set
\[
    s_0=-\log z_0.
\]
Choose the admissible control
\[
    u(t)=a(s_0+t,c),
    \qquad t\geq0.
\]
Since \(a(s_0+\cdot,c)\) is measurable and takes values in \(\{-1,1\}\), this
control is admissible, namely \(u\in L^\infty_{\mathrm{loc}}([0,\infty),[-1,1])\). The first two equations of \eqref{eq:03.02} give
\[
    z(t)=z_0e^{-t},
    \qquad
    c(t)=c.
\]
For the \(y\)-coordinate, using \(y(0)=F(z_0,c)=z_0q(s_0,c)\), we have
\[
    \frac{d}{dt}\bigl(e^t y(t)\bigr)
    =
    e^t z(t)u(t)
    =
    z_0u(t)
\]
for a.e. \(t\geq0\). Hence
\[
    y(t)
    =
    e^{-t}
    \left(
        z_0q(s_0,c)
        +z_0\int_0^t a(s_0+r,c)\,dr
    \right).
\]
By the definition of \(q\),
\[
    q(s_0+t,c)
    =
    q(s_0,c)+\int_0^t a(s_0+r,c)\,dr.
\]
Therefore
\[
    y(t)
    =
    z_0e^{-t}q(s_0+t,c).
\]
Since
\[
    -\log(z_0e^{-t})=s_0+t,
\]
we obtain
\[
    y(t)
    =
    F(z_0e^{-t},c)
    =
    F(z(t),c(t)).
\]
Thus
\[
    (z(t),c(t),y(t))
    =
    (z(t),c,F(z(t),c))
    \in Q
\]
for all \(t\geq0\).

It remains to consider \(z_0=0\). In this case \(F(0,c)=0\), so
\(x_0=(0,c,0)\). Choose, for instance, \(u(t)\equiv0\). Then \(z(t)=0\),
\(c(t)=c\), and the \(y\)-equation reduces to \(\dot y=-y\) with \(y(0)=0\), so
\(y(t)=0\) for all \(t\geq0\). Hence
\[
    (z(t),c(t),y(t))=(0,c,0)=(0,c,F(0,c))\in Q
\]
for all \(t\geq0\). This proves weak invariance of \(Q\).
\end{proof}

The next corollary is the special case, starting from \(K\), used in the exact entropy count.

\begin{corollary}[Matching controls stay on the graph]
\label{cor:matching_control_stays_on_graph}
Let \(c\in C\), and let \(u\in L^\infty_{\mathrm{loc}}([0,\infty),[-1,1])\)
satisfy \(u(t)=a(t,c)\) for a.e. \(t\geq 0\). Then the solution of \eqref{eq:03.02} with initial condition \((z(0),c(0),y(0))=(1,c,0)\) satisfies \((z(t),c(t),y(t))\in Q\) for all \(t\geq 0\).
\end{corollary}

\begin{proof}
This is the case \(z_0=1\) of Proposition~\ref{prop:Q_weakly_invariant}. Then
\(s_0=-\log z_0=0\), so the matching control constructed there is
\(a(t,c)\). Changing a control on a null set does not change the corresponding
Carathéodory solution, and hence the same conclusion holds for every
\(u=a(\cdot,c)\) a.e.
\end{proof}

\begin{lemma}[Finite-horizon matching]
\label{lem:finite-horizon-matching}
Let \(T>0\), let \(c\in C\), and let \(u\in\mathcal U\) satisfy \(u(t)=a(t,c)\) for a.e. \(t\in[0,T]\). Then the solution of \eqref{eq:03.02} with initial condition \((z(0),c(0),y(0))=(1,c,0)\) satisfies \((z(t),c(t),y(t))\in Q\) for every \(t\in[0,T]\).
\end{lemma}

\begin{proof}
For \(0\leq t\leq T\), the first two state equations give \(z(t)=e^{-t}\) and \(c(t)=c\), while variation of constants yields
\[
    y(t)
    =
    e^{-t}\int_0^t u(s)\,ds
    =
    e^{-t}\int_0^t a(s,c)\,ds
    =
    F(e^{-t},c).
\]
Hence \((z(t),c(t),y(t))=(e^{-t},c,F(e^{-t},c))\in Q\) for every \(t\in[0,T]\).
\end{proof}

\begin{proposition}[The graph forces symbolic matching]
\label{prop:graph_forces_symbolic_matching}
Let \(T>0\), let \(c\in C\), and let
\(u=\tilde u|_{[0,T]}\) be the restriction to \([0,T]\) of an admissible control \(\tilde u\in\mathcal U\). Consider the solution of \eqref{eq:03.02}
with initial condition \((z(0),c(0),y(0))=(1,c,0)\). If \((z(t),c(t),y(t))\in Q\) for all \(t\in[0,T]\), then \(u(t)=a(t,c)\) for a.e. \(t\in[0,T]\). Consequently, any mismatch between \(u\) and \(a(\cdot,c)\) on a set of
positive measure prevents the trajectory from remaining in \(Q\) on the whole
interval \([0,T]\).
\end{proposition}

\begin{proof}
As before,
\[
    z(t)=e^{-t},
    \qquad
    c(t)=c,
\]
and the solution of the \(y\)-equation satisfies
\[
    y(t)
    =
    e^{-t}\int_0^t u(s)\,ds .
\]
On the other hand, the assumption that the trajectory remains in \(Q\) gives
\[
    y(t)
    =
    F(z(t),c)
    =
    F(e^{-t},c)
    =
    e^{-t}q(t,c)
    =
    e^{-t}\int_0^t a(s,c)\,ds
\]
for every \(t\in[0,T]\). Hence
\[
    \int_0^t u(s)\,ds
    =
    \int_0^t a(s,c)\,ds
    \qquad
    \text{for every }t\in[0,T].
\]
Equivalently,
\[
    \int_0^t \bigl(u(s)-a(s,c)\bigr)\,ds=0
    \qquad
    \text{for every }t\in[0,T].
\]
The function \(u-a(\cdot,c)\) belongs to \(L^\infty([0,T])\). By the
fundamental theorem of calculus for absolutely continuous functions, the
identity above implies
\[
    u(t)=a(t,c)
    \qquad
    \text{for a.e. }t\in[0,T].
\]

The final statement follows immediately: if \(u\) differs from \(a(\cdot,c)\)
on a set of positive measure in \([0,T]\), then the equality above cannot hold,
and therefore the trajectory cannot remain in \(Q\) for every \(t\in[0,T]\).
\end{proof}

\subsection{Fading Estimate}
\label{subsection:fading_estimate}

We now prove the estimate used in the approximate invariance argument. The
point is that a mismatch occurring after time \(N\) can affect the
\(y\)-coordinate only through the exponentially decaying factor \(e^{-t}\).

\begin{proposition}[Fading estimate]
Let \(c\in C\), \(N\in\mathbb N\), and \(u\in L^\infty_{\mathrm{loc}}([0,\infty),\allowbreak [-1,1])\). Consider the solution of \eqref{eq:03.02} with initial condition \((z(0),c(0),y(0))=(1,c,0)\). Assume that \(u(t)=a(t,c)\) for a.e. \(t\in[0,N]\). Distances below are computed with respect to the Euclidean metric on \(M=\mathbb R^3\). Then, for every \(t\geq0\),
\begin{equation}
    \operatorname{dist}\bigl((z(t),c(t),y(t)),Q\bigr)
    \leq
    \frac{2}{e}e^{-N}.
\end{equation}
More precisely,
\begin{equation}
    \operatorname{dist}\bigl((z(t),c(t),y(t)),Q\bigr)=0
    \qquad
    \text{for }0\leq t\leq N,
\end{equation}
and
\begin{equation}
    \operatorname{dist}\bigl((z(t),c(t),y(t)),Q\bigr)
    \leq
    2(t-N)e^{-t}
    \qquad
    \text{for }t\geq N.
\end{equation}
\end{proposition}

\begin{proof}
As before,
\[
    z(t)=e^{-t},
    \qquad
    c(t)=c,
\]
and the \(y\)-coordinate satisfies
\[
    y(t)
    =
    e^{-t}\int_0^t u(s)\,ds .
\]
On the other hand, by the definition of \(F\),
\[
    F(z(t),c)
    =
    F(e^{-t},c)
    =
    e^{-t}q(t,c)
    =
    e^{-t}\int_0^t a(s,c)\,ds .
\]
Therefore
\[
    y(t)-F(z(t),c)
    =
    e^{-t}\int_0^t \bigl(u(s)-a(s,c)\bigr)\,ds .
\]
Since
\[
    (z(t),c,F(z(t),c))\in Q,
\]
we have
\[
    \operatorname{dist}\bigl((z(t),c,y(t)),Q\bigr)
    \leq
    |y(t)-F(z(t),c)|.
\]

If \(0\leq t\leq N\), then \(u=a(\cdot,c)\) a.e. on \([0,t]\). Hence
\[
    y(t)-F(z(t),c)=0,
\]
and the trajectory lies on \(Q\).

Now let \(t\geq N\). Since the two controls agree a.e. on \([0,N]\), we get
\[
    y(t)-F(z(t),c)
    =
    e^{-t}\int_N^t \bigl(u(s)-a(s,c)\bigr)\,ds .
\]
Both \(u(s)\) and \(a(s,c)\) take values in \([-1,1]\), and therefore
\[
    |u(s)-a(s,c)|\leq 2
    \qquad
    \text{for a.e. }s.
\]
It follows that
\[
    |y(t)-F(z(t),c)|
    \leq
    2e^{-t}(t-N).
\]
Finally, for \(t\geq N\), we have
\[
    2e^{-t}(t-N)
    =
    2e^{-N}(t-N)e^{-(t-N)}
    \leq
    \frac{2}{e}e^{-N},
\]
because \(\sup_{x\geq0}xe^{-x}=e^{-1}\). This proves the estimate.
\end{proof}

\begin{corollary}[Uniform approximation after a finite prefix]
\label{cor:fading_prefix}
For every \(\varepsilon>0\), there exists \(N\in\mathbb N\) such that the
following holds. If \(c\in C\) and
\(u\in L^\infty_{\mathrm{loc}}([0,\infty),[-1,1])\) satisfies \(u(t)=a(t,c)\) for a.e. \(t \in [0,N]\), then the corresponding trajectory satisfies
\begin{equation}
    \operatorname{dist}\bigl((z(t),c(t),y(t)),Q\bigr)
    <\varepsilon
    \qquad
    \text{for all }t\geq0.
\end{equation}
\end{corollary}

\begin{proof}
Choose \(N\in\mathbb N\) so large that
\[
    \frac{2}{e}e^{-N}<\varepsilon .
\]
The conclusion follows directly from the fading estimate.
\end{proof}

\subsection{Hausdorff Convergence}
\label{subsection:hausdorff_convergence}

We verify the Hausdorff convergence statements used in the construction.
Recall that, for compact subsets \(A,B\) of a metric space,
\begin{equation}
    d_H(A,B)
    :=
    \max\left\{
        \sup_{a\in A}\operatorname{dist}(a,B),
        \sup_{b\in B}\operatorname{dist}(b,A)
    \right\}.
\end{equation}

\begin{proposition}[Hausdorff convergence of finite Cantor skeletons]\label{prop:Cn_Hausdorff_convergence}
The finite sets \(C_n\) converge to \(C\) in the Hausdorff metric. More precisely, \(d_H(C_n,C)\leq 3^{-n}\).
\end{proposition}

\begin{proof}
Since \(C_n\subseteq C\), we have
\[
    \sup_{c_n\in C_n}\operatorname{dist}(c_n,C)=0.
\]
It remains to estimate the distance from points of \(C\) to \(C_n\). Let
\(c\in C\), and set
\[
    \omega
    :=
    \bigl(\sigma_1(c),\ldots,\sigma_n(c)\bigr)
    \in\{-1,1\}^n.
\]
By construction, \(c_\omega^+\in C_n\), and the symbolic codes of \(c\) and
\(c_\omega^+\) agree in their first \(n\) entries. Therefore
\[
    \operatorname{dist}(c,C_n)
    \leq
    |c-c_\omega^+|.
\]
Using the definition of \(\pi\), we obtain
\[
    |c-c_\omega^+|
    \leq
    \sum_{j=n+1}^{\infty}\frac{2}{3^j}
    =
    3^{-n}.
\]
Taking the supremum over \(c\in C\), we obtain
\[
    \sup_{c\in C}\operatorname{dist}(c,C_n)
    \leq
    3^{-n}.
\]
Combining the two estimates gives \(d_H(C_n,C)\leq 3^{-n}\). In particular, \(C_n\to C\) in the Hausdorff metric.
\end{proof}

\begin{corollary}[Hausdorff convergence of the initial sets]\label{cor:Kn_Hausdorff_convergence}
The sets \(K_n\) converge to \(K\) in the Hausdorff metric.
\end{corollary}

\begin{proof}
Since \(K=\{(1,c,0):c\in C\}\), \(K_n=\{(1,c,0):c\in C_n\}\), and, for all \(c,c'\in C\), \(|(1,c,0)-(1,c',0)|
    =
    |c-c'|\), we have \(d_H(K_n,K)
    =
    d_H(C_n,C)\). By the previous proposition, \(d_H(C_n,C)\leq 3^{-n}\). Hence \(d_H(K_n,K)\leq 3^{-n}\), and consequently \(K_n\to K\) in the Hausdorff metric.
\end{proof}

\subsection{Proofs of Theorems}
\label{subsection:proofs-of-theorems}

\begin{proof}[Proof of Theorem~\ref{thm:main-separation}]
Fix \(T>0\) and put \(m=\lceil T\rceil\). We first prove
\(r_{\mathrm{inv}}^*(T,K,Q)\allowbreak \ge 2^m\).
For each word \(\omega\in\{-1,1\}^m\), choose any \(c\in C\)
whose first \(m\) symbols are \(\omega\).
If one control \(u\) kept two such points, with prefixes \(\omega\) and \(\omega'\), in \(Q\),
then Proposition~\ref{prop:graph_forces_symbolic_matching} would imply
\(u=a(\cdot,c)=a(\cdot,c')\) a.e. on \([0,T]\), where \(c\) and \(c'\) have prefixes \(\omega\) and \(\omega'\), respectively.
If \(\omega\ne\omega'\), let \(j\) be the first differing index.
Since \(j\le m=\lceil T\rceil\), the interval
\([j-1,j)\cap[0,T]\) has positive measure. On that interval the two
reference signals take opposite values, a contradiction.
Hence at least \(2^m\) controls are necessary.

Conversely, for each word \(\omega\in\{-1,1\}^m\), choose an admissible control which realizes \(\omega\) on the first \(m\) unit intervals and is arbitrary afterwards. For each \(c\in C\), assign the control whose first \(m\) symbols agree with those of \(c\). Since \(T\leq m\), this control agrees with \(a(\cdot,c)\) a.e. on \([0,T]\); Lemma~\ref{lem:finite-horizon-matching} therefore keeps the trajectory in \(Q\) on \([0,T]\). Thus
\(r_{\mathrm{inv}}^*(T,K,Q)=2^m\), and so
\(h_{\mathrm{inv}}^*(K,Q)=\log2\).

For ordinary invariance entropy, fix \(\varepsilon>0\). Choose \(N\) as in
Corollary~\ref{cor:fading_prefix}. The \(2^N\) prefix controls form a
\((T,\varepsilon,K,Q)\)-spanning family for every \(T>0\). Hence
\(r_{\mathrm{inv}}(T,\varepsilon,K,Q)\le2^N\) for all \(T\), while
\(r_{\mathrm{inv}}(T,\varepsilon,K,Q)\ge1\). Therefore
\(h_{\mathrm{inv}}(\varepsilon,K,Q)=0\). Since this holds for all
\(\varepsilon>0\), \(h_{\mathrm{inv}}(K,Q)=0\).
\end{proof}

\begin{proof}[Proof of Theorem~\ref{thm:hausdorff-failure}]
By Corollary~\ref{cor:Kn_Hausdorff_convergence}, \(K_n\to K\) in the Hausdorff metric. Fix \(n\). Since \(K_n\) consists of \(2^n\) points, and each point has a prescribed symbolic code, Corollary~\ref{cor:matching_control_stays_on_graph} gives a family of at most \(2^n\) controls which is \((T,K_n,Q)\)-spanning for every \(T>0\). Hence \(1\le r_{\mathrm{inv}}^*(T,K_n,Q)\le 2^n\) for all \(T>0\), and therefore \(h_{\mathrm{inv}}^*(K_n,Q)=\limsup_{T\to\infty}T^{-1}\log r_{\mathrm{inv}}^*(T,K_n,Q)=0\). On the other hand, Theorem~\ref{thm:main-separation} gives \(h_{\mathrm{inv}}^*(K,Q)=\log2\). Thus \(\liminf_{n\to\infty}h_{\mathrm{inv}}^*(K_n,Q)<h_{\mathrm{inv}}^*(K,Q)\).
\end{proof}

\section{Discussion}

The construction should be read against two different backgrounds. Existing positive formulas typically tie invariance entropy to unstable growth or to robust hyperbolic structure, whereas the present example stores information in an exact geometric matching constraint. The Appendix~\ref{appendix:ordinary-linear-comparison} separately shows that ordinary Hausdorff semicontinuity can fail through a conventional expansion mechanism.

\subsection{Positive Regimes for Kawan's Principle}

For linear systems and several smooth or hyperbolic control settings, invariance entropy is described or bounded by unstable eigenvalues, positive Lyapunov data, derivative or volume growth, and unstable determinants \cite{ColoniusKawan2009,Kawan2011ControlSets,Kawan2011UpperLower,Kawan2011LowerBounds,Kawan2013,DaSilvaKawan2016Hyperbolic}. Robust continuity results for uniformly hyperbolic chain control sets additionally use a persistent stable--unstable splitting and structural stability \cite{DaSilvaKawan2016Robustness}. The present compact pair \((K,Q)\) has none of these expansion, accessibility, isolation, or hyperbolicity features, so the counterexample does not contradict those results.

\subsection{Methodological Difference}
\label{subsec:methodological-difference}

Symbolic descriptions, control
words, and finite spanning families already appear in topological feedback
entropy and invariance entropy, where they encode the number of messages or
controls needed to achieve invariance \cite{NairEvansMareelsMoran2004,ColoniusKawan2009}.  Exponential decay and
exponential weights also appear naturally in stabilization and estimation
problems \cite{Colonius2012MinimalBitRates,Kawan2017StateEstimation,ColoniusHamzi2021}.  What is special in our construction is not any one of these ingredients in
isolation, but the way they are combined to separate exact invariance from
fixed-tolerance invariance.

The construction has three roles. The Cantor coordinate stores an infinite binary sequence, the graph \(Q\) forces exact reproduction of that sequence, and the factor \(z(t)=e^{-t}\) exponentially suppresses the geometric effect of late mismatches. Proposition~\ref{prop:graph_forces_symbolic_matching} and Corollary~\ref{cor:fading_prefix} formalize these two complementary effects.

The absence of an expansion mechanism can also be verified directly from the variational system. For a fixed admissible control \(u\), linearization of \eqref{eq:03.02} with respect to the state gives
\[
\frac{d}{dt}
\begin{pmatrix}
\delta z\\[1mm]
\delta c\\[1mm]
\delta y
\end{pmatrix}
=
\begin{pmatrix}
-1&0&0\\
0&0&0\\
u(t)&0&-1
\end{pmatrix}
\begin{pmatrix}
\delta z\\[1mm]
\delta c\\[1mm]
\delta y
\end{pmatrix}.
\]
Its fundamental matrix is
\[
\Phi_u(t)
=
\begin{pmatrix}
e^{-t}&0&0\\
0&1&0\\
e^{-t}\displaystyle\int_0^t u(s)\,ds&0&e^{-t}
\end{pmatrix}.
\]
Because \(|u|\leq1\), the \((z,y)\)-block is bounded by a constant multiple of \((1+t)e^{-t}\), while the \(c\)-direction is neutral. Thus the exponential rates are \(-1,0,-1\), in particular there is no positive Lyapunov exponent for any admissible control. The positive strict entropy in Theorem~\ref{thm:main-separation} is therefore produced by exact symbolic matching in the thin target graph, not by exponential separation of nearby state trajectories.

Theorem~\ref{thm:main-separation} already contains the complete entropy calculation: strict invariance reads one additional symbol per unit time, whereas every fixed positive tolerance sees only a finite prefix. We do not repeat that calculation here.

\subsection{Ordinary Invariance Entropy}
\label{subsection:ordinary_invariance_entropy}

For the Cantor construction, ordinary invariance entropy is zero both for \(K\) and for every finite approximant \(K_n\). The first assertion is part of Theorem~\ref{thm:main-separation}; the second follows because a finite catalogue containing one all-time matching control for each point of \(K_n\) is spanning for every horizon. Hence the strict Hausdorff discontinuity proved above does not produce an ordinary-entropy discontinuity for the same sequence.

Ordinary Hausdorff semicontinuity may nevertheless fail for a different reason. Appendix~\ref{appendix:ordinary-linear-comparison} gives the elementary expanding system \(\dot x=x+u\), in which initial uncertainty grows at rate \(e^T\). That comparison separates the conventional expansion mechanism from the zero-scale symbolic mechanism of the main construction.

\subsection{Data-rate Interpretation}
\label{subsec:data-rate-interpretation}

Let \[B_0(T)=\log_2 r_{\mathrm{inv}}^*(T,K,Q),\]
and let \[
B_\varepsilon(T)=\log_2 r_{\mathrm{inv}}(T,\varepsilon,K,Q)
\] for \(\varepsilon>0\). Theorem~\ref{thm:main-separation} and Corollary~\ref{cor:fading_prefix} imply
\[
\limsup_{T\to\infty}\frac{B_0(T)}{T}=1,
\qquad
\limsup_{T\to\infty}\frac{B_\varepsilon(T)}{T}=0
\quad(\varepsilon>0).
\]
Thus exact invariance requires a persistent stream of one bit per unit time, while any fixed positive tolerance requires only a finite initial message. This operational distinction is the data-rate form of the strict--ordinary separation; it does not rely on a positive state-space Lyapunov exponent.

\section{Conclusion}

This paper provides counterexamples to two long-standing stability expectations for invariance entropy in the absence of additional structural hypotheses. We construct a control system with a compact invariant target pair showing that finite strict invariance entropy does not force equality between strict and ordinary invariance entropy, and that strict invariance entropy need not be lower semicontinuous under unconditional Hausdorff convergence of compact initial sets, with the target set fixed. Thus the expected compatibility between exact invariance, approximate invariance, and finite compact approximation fails in full generality.

The example identifies a zero-scale source of information complexity. The positive strict rate is not produced by dynamical expansion, but by an exact matching constraint encoded in the invariant geometry. This constraint persists in the exact infinite-horizon problem, while it disappears at every fixed positive tolerance and is missed by every finite approximation of the initial set. In this sense, ordinary invariance entropy can collapse symbolic information that remains essential for strict invariance.

These counterexamples also clarify the role of the structural assumptions used in positive results on invariance entropy. Conditions such as isolation, accessibility, hyperbolicity, or expansion-type formulas are not merely technical devices; without additional structure, exact invariance may retain information that approximate invariance cannot detect. A natural next step is to identify sharp hypotheses under which finite strict entropy agrees with ordinary entropy and under which Hausdorff perturbations preserve semicontinuity. Another direction is to develop scale-sensitive invariants that keep track of the joint dependence on time and tolerance before the ordinary entropy limit erases this hidden information.

\section*{Acknowledgments}

The author is deeply grateful to Prof. Lei Guo (AMSS), the author's doctoral advisor, for his guidance and intellectual influence. His work on the quantitative limitations of feedback for uncertain systems strongly shaped the author's interest in information-limited control. In the course of exploring this viewpoint, the author was led to the invariance-entropy and data-rate theory developed by Colonius, Kawan, and their collaborators, and the questions studied in this paper came to the author's attention from that broader perspective. The author also wishes to thank his fellow students Junhe Gong (AMSS) and Kainan Guo (AMSS) for reading an earlier draft of this paper and carefully checking the mathematical proofs.

\appendix

\section{Ordinary Linear Comparison}
\label{appendix:ordinary-linear-comparison}

This appendix records a simple comparison example for ordinary invariance
entropy.  Its purpose is not to provide a new entropy computation: the value
computed below is the one-dimensional case of the linear-system formula of
Colonius and Kawan; see \cite[Theorem~5.1]{ColoniusKawan2009} and
\cite[Section~4.1]{KawanThesis2009}.  The point is only to make explicit the
Hausdorff semicontinuity consequence which is used for comparison in
Subsection~\ref{subsection:ordinary_invariance_entropy}.

\begin{proposition}[Ordinary Linear Comparison]
\label{prop:ordinary-linear-comparison}
Consider the following one-dimensional control system
\begin{equation}
    \dot x=x+u,
    \qquad
    u\in[-1,1],
\end{equation}
on \(M=\mathbb{R}\), and let \(Q=[-1,1]\).  Then \(Q\) is weakly invariant.  If
\(K\subseteq Q\) is compact and has positive Lebesgue measure, then
\begin{equation}
    h_{\mathrm{inv}}(K,Q)=1 .
    \label{eq:ordinary-linear-positive-measure}
\end{equation}
If \(P\subseteq Q\) is finite and nonempty, then
\begin{equation}
    h_{\mathrm{inv}}(P,Q)=0 .
    \label{eq:ordinary-linear-finite-set}
\end{equation}
Consequently, for fixed \(Q=[-1,1]\), the map
\(K\mapsto h_{\mathrm{inv}}(K,Q)\) is neither lower nor upper semicontinuous
with respect to Hausdorff convergence of compact subsets \(K\subseteq Q\).
\end{proposition}

\begin{proof}
First, \(Q\) is weakly invariant.  Indeed, for \(x_0\in Q\), the constant
control \(u(t)\equiv -x_0\) is admissible and the corresponding solution is
identically equal to \(x_0\).  Hence the trajectory remains in \(Q\) for all
positive times.

We next prove \eqref{eq:ordinary-linear-positive-measure}.  For any admissible
control \(u\), the solution is given by
\[
    \varphi(t,x,u)
    =
    e^t x+\int_0^t e^{t-s}u(s)\,ds .
\]

Fix \(T>0\) and \(\varepsilon>0\).  Choose an
\(\varepsilon e^{-T}/2\)-net \(\{q_1,\ldots,q_N\}\subseteq Q\) for \(Q\).  Since
\(Q\) has length \(2\), we may choose such a net with
\[
    N\le \frac{4e^T}{\varepsilon}+2 .
\]
For each \(q_j\), choose the constant control \(u_j(t)\equiv -q_j\).  Since
\(q_j\in[-1,1]\), this control is admissible.  Under this control one has
\[
    \varphi(t,x,u_j)
    =
    q_j+e^t(x-q_j).
\]
Thus, for every \(x\in Q\), there exists \(q_j\) such that \[
|x-q_j|\le \frac{\varepsilon e^{-T}}{2}. 
\] For this index \(j\), and for every \(0\le t\le T\), we obtain \[ 
\operatorname{dist}(\varphi(t,x,u_j),Q) \le |\varphi(t,x,u_j)-q_j| = e^t|x-q_j| \le e^T\frac{\varepsilon e^{-T}}{2} = \frac{\varepsilon}{2} < \varepsilon . 
\] Hence \(\varphi(t,x,u_j)\in N_\varepsilon(Q)\) for all \(0\le t\le T\). Therefore the controls \(u_1,\ldots,u_N\) form a
\((T,\varepsilon)\)-spanning family for \((K,Q)\), for every compact
\(K\subseteq Q\).  Hence
\[
    r_{\mathrm{inv}}(T,\varepsilon,K,Q)
    \le
    \frac{4e^T}{\varepsilon}+2 .
\]
Taking the exponential growth rate in \(T\), we obtain
\[
    h_{\mathrm{inv}}(\varepsilon,K,Q)\le 1 .
\]

Now assume that \(K\subseteq Q\) is compact and has positive Lebesgue measure.
We prove the reverse inequality.  Fix a single admissible control \(u\), and
write
\[
    b_u(T):=\int_0^T e^{T-s}u(s)\,ds .
\]
At time \(T\), the solution satisfies
\[
    \varphi(T,x,u)=e^T x+b_u(T).
\]
Let \(A_u(T,\varepsilon)\) denote the set of all \(x\in K\) for which this
single control keeps the trajectory in \(N_\varepsilon(Q)\) on the whole
interval \([0,T]\).  Then, in particular, the terminal condition must hold:
\[
    e^T x+b_u(T)
    =
    \varphi(T,x,u)
    \in
    (-1-\varepsilon,1+\varepsilon).
\]
Therefore every \(x\in A_u(T,\varepsilon)\) belongs to the interval
\[
    e^{-T}\bigl((-1-\varepsilon,1+\varepsilon)-b_u(T)\bigr),
\]
whose length is \(2(1+\varepsilon)e^{-T}\).  Hence
\[
    \lambda^*\bigl(A_u(T,\varepsilon)\bigr)
    \le
    2(1+\varepsilon)e^{-T},
\]
where \(\lambda^*\) denotes Lebesgue outer measure and \(\lambda\) denotes Lebesgue measure.

Let \(S\) be any finite \((T,\varepsilon)\)-spanning family for \((K,Q)\).  By
definition of \(A_u(T,\varepsilon)\), the sets \(A_u(T,\varepsilon)\) with
\(u\in S\) cover \(K\):
\[
    K
    \subseteq
    \bigcup_{u\in S} A_u(T,\varepsilon).
\]
Since \(K\) is compact and hence Lebesgue measurable, outer measure
subadditivity gives
\[
    \lambda(K)
    \le
    \sum_{u\in S}
    \lambda^*\bigl(A_u(T,\varepsilon)\bigr)
    \le
    \#S\cdot 2(1+\varepsilon)e^{-T}.
\]
Consequently every finite \((T,\varepsilon)\)-spanning family satisfies
\[
    \#S
    \ge
    \frac{\lambda(K)}{2(1+\varepsilon)}e^T .
\]
Taking the minimum over all finite \((T,\varepsilon)\)-spanning families yields
\[
    r_{\mathrm{inv}}(T,\varepsilon,K,Q)
    \ge
    \frac{\lambda(K)}{2(1+\varepsilon)}e^T .
\]
Thus
\[
    h_{\mathrm{inv}}(\varepsilon,K,Q)\ge 1 .
\]
Together with the upper bound, this proves
\[
    h_{\mathrm{inv}}(\varepsilon,K,Q)=1
    \quad\text{for every }\varepsilon>0.
\]
Letting \(\varepsilon\downarrow0\) proves
\eqref{eq:ordinary-linear-positive-measure}.

We now prove \eqref{eq:ordinary-linear-finite-set}. Let \(P=\{x_1,\ldots,x_N\}\subseteq Q\) be finite and nonempty.  For each \(x_j\),
choose the constant control \(u_j(t)\equiv -x_j\).  The trajectory starting from
\(x_j\) is then constant and remains in \(Q\) for all positive times.  Therefore
the \(N\) controls \(u_1,\ldots,u_N\) form a \((T,\varepsilon)\)-spanning family
for \((P,Q)\) for every \(T>0\) and every \(\varepsilon>0\).  Hence
\[
    r_{\mathrm{inv}}(T,\varepsilon,P,Q)
    \le
    N
    \quad\text{for all }T>0,\ \varepsilon>0.
\]
It follows that
\[
    h_{\mathrm{inv}}(\varepsilon,P,Q)=0
    \quad\text{for every }\varepsilon>0,
\]
and therefore \(h_{\mathrm{inv}}(P,Q)=0\).

It remains to spell out the Hausdorff semicontinuity consequences.  Define
\[
    E_n
    :=
    \left\{
        -1+\frac{2j}{n}:j=0,\ldots,n
    \right\}.
\]
Then \(E_n\subseteq Q\), and every point of \(Q=[-1,1]\) lies within distance
\(1/n\) of some point of \(E_n\).  Hence
\[
    d_H(E_n,Q)\le \frac{1}{n}\longrightarrow0.
\]
By \eqref{eq:ordinary-linear-finite-set} and
\eqref{eq:ordinary-linear-positive-measure}, we have
\begin{equation}
    h_{\mathrm{inv}}(E_n,Q)=0
    \quad\text{for all }n,
    \qquad
    h_{\mathrm{inv}}(Q,Q)=1 .
    \label{eq:ordinary-not-lsc}
\end{equation}
If \(K\mapsto h_{\mathrm{inv}}(K,Q)\) were lower semicontinuous at \(Q\), then
\[
    h_{\mathrm{inv}}(Q,Q)
    \le
    \liminf_{n\to\infty}h_{\mathrm{inv}}(E_n,Q),
\]
which contradicts \eqref{eq:ordinary-not-lsc}.  Thus lower semicontinuity
fails.

For upper semicontinuity, let
\[
    J_n=[-1/n,1/n],
    \qquad
    J=\{0\}.
\]
Then \(J_n\subseteq Q\), \(J\subseteq Q\), and
\[
    d_H(J_n,J)=\frac{1}{n}\longrightarrow0.
\]
Each \(J_n\) has positive Lebesgue measure, whereas \(J=\{0\}\) is finite.
Therefore \eqref{eq:ordinary-linear-positive-measure} and
\eqref{eq:ordinary-linear-finite-set} imply
\begin{equation}
    h_{\mathrm{inv}}(J_n,Q)=1
    \quad\text{for all }n,
    \qquad
    h_{\mathrm{inv}}(J,Q)=0 .
    \label{eq:ordinary-not-usc}
\end{equation}
If \(K\mapsto h_{\mathrm{inv}}(K,Q)\) were upper semicontinuous at \(J=\{0\}\),
then
\[
    h_{\mathrm{inv}}(J,Q)
    \ge
    \limsup_{n\to\infty}h_{\mathrm{inv}}(J_n,Q),
\]
which contradicts \eqref{eq:ordinary-not-usc}.  Thus upper semicontinuity also
fails.
\end{proof}

The example above should not be confused with the main construction of the
paper.  Its mechanism is expansion: initial uncertainty is amplified at rate
\(e^T\).  The Cantor construction in the main text has the opposite feature:
tail discrepancies are exponentially faded by the factor \(z(t)=e^{-t}\).
Thus the linear example only illustrates that ordinary invariance entropy may
also fail Hausdorff semicontinuity, while the main construction proves the
strict--ordinary separation and the failure of Hausdorff lower semicontinuity
for strict invariance entropy.

\bibliographystyle{siamplain}
\bibliography{references}

\end{document}